\documentclass[11pt,leqno]{article}
\usepackage[utf8]{inputenc}
\usepackage{amsmath,amssymb,amsthm,latexsym}
\usepackage{indentfirst}
\usepackage{amsmath}

\newtheorem{theorem}{\bf \large Theorem}[section]
\newtheorem{proposition}{\bf \large Proposition}[section]
\newtheorem{corollary}{\bf \large Corollary}[section]
\newtheorem{lemma}{\bf \large Lemma}[section]
\newtheorem{example}{\bf \large Example}[section]
\newtheorem{remark}{\bf \large Remark}[section]

\title{Examples of topological sphere  $\lambda$-self-expander}

\author {{Mengdan Qi} \\
\small{Department of Mathematics, Beijing Institute of Technology, Beijing, 100081, China.} \\
\small{ E-mail: mengdanqi@bit.edu.cn.}}
\date{}

\begin{document}
\maketitle
\begin{abstract}
A $\lambda$-self-expander $F: M^n\to \mathbb{R}^{n+1}$ is  the solution of the isoperimetric problem of weight $e^{\frac{|x|^2}{4}}$. In this paper, for a sufficiently large $\lambda$,  we construct an immersed non-embedded spherical $\lambda$-self-expander in $\mathbb{R}^{n+1}$ with the mean curvature $H>0$.  Moreover, we obtain the explicit solution for $n=2$.
\end{abstract}


\section{Introduction}

An immersed hypersurface $F: M^n \to \mathbb{R}^{n+1}$ is called a $\lambda$-self-expander if it satisfies
\begin{equation}\label{def}
    H - \frac12 \langle F,\xi \rangle = \lambda,
\end{equation}
where $\xi$ is  the inward unit normal vector field of $F$,  $H$ is the mean curvature of $F$, and $\lambda$ is a constant.
The concept of $\lambda$-self-expanders was first proposed by Ancari and Cheng in \cite{ac}. When $\lambda=0$, the equation reduces to the defining equation of standard self-expanders, so $\lambda$-self-expanders serve as a natural generalization of classical self-expanders. Moreover, $\lambda$-self-expanders are critical points of a weighted area functional under volume constraint, with the weight function given by $e^{\frac{|F|^2}{4}}$.
These hypersurfaces are dual to the well-studied $\lambda$-hypersurfaces, which satisfy $H+\tfrac12\langle F,\xi\rangle=\lambda$ and correspond to self-shrinking weighted solitons. Despite their natural variational origin, the global rigidity and classification theory for  closed $\lambda$-self-expanders remains significantly underdeveloped compared to the mature literature on $\lambda$-hypersurfaces, especially, we know very little about examples of closed $\lambda$-self-expander.

The construction and classification of self-similar hypersurfaces have been extensively studied. For self-shrinkers, the simplest examples are hyperplanes through the origin, round spheres and generalized cylinders. In 1992, Angenent \cite{ang} constructed an embedded self-shrinker which is diffeomorphic to $\mathbb S^1\times\mathbb S^{n-1}$. Chopp \cite{chopp} subsequently developed a numerical algorithm and obtained numerical evidence for several further self-shrinking examples. Further embedded self-shrinkers, including noncompact and higher-genus examples, were constructed by Kleene and M{\o}ller \cite{km}, Kapouleas, Kleene and M{\o}ller \cite{kkm}, M{\o}ller \cite{moller}, and Nguyen \cite{ng3}.

In the immersed setting, further nontrivial examples are known. Drugan \cite{drugan} constructed an immersed and non-embedded $\mathbb S^2$ self-shrinker in $\mathbb R^3$, and Drugan and Kleene \cite{dk} constructed further immersed rotational self-shrinkers. Drugan and Nguyen \cite{dn} gave an alternative proof of the existence of a self-shrinking torus, while Drugan, Lee and Nguyen \cite{dln} surveyed known results on closed self-shrinkers and discussed techniques used in the construction of closed self-shrinkers with rotational symmetry. These results provide several nontrivial compact and complete self-shrinkers beyond the standard examples.

A closely related class is that of $\lambda$-hypersurfaces. Cheng and Wei \cite{cw1} introduced $\lambda$-hypersurfaces in their study of the weighted volume-preserving mean curvature flow. The case $\lambda=0$ reduces to the self-shrinker equation. The variational, stability and rigidity properties of $\lambda$-hypersurfaces have subsequently been studied by many authors. McGonagle and Ross \cite{mr} considered the corresponding Gaussian isoperimetric problem. Rigidity and gap phenomena for complete $\lambda$-hypersurfaces were studied, among others, by Cheng, Ogata and Wei \cite{cow}, Guang \cite{guang}, Wang, Xu and Zhao \cite{wxz}, and Zhu, Fang and Chen \cite{zfc}. Cheng and Wei also studied stability, area growth and other global geometric properties of $\lambda$-hypersurfaces \cite{cw3}.

Besides rigidity results, a number of nontrivial examples of $\lambda$-hypersurfaces are known. Cheng and Wei \cite{cw2} constructed compact rotational $\lambda$-hypersurfaces, including both embedded and immersed examples. Ross \cite{ross} constructed closed embedded $\lambda$-hypersurfaces diffeomorphic to $\mathbb S^{n-1}\times\mathbb S^{n-1}\times\mathbb S^1$. Related to the present work, Li and Wei \cite{liwei} constructed an immersed and non-embedded $\mathbb S^n$ $\lambda$-hypersurface in $\mathbb R^{n+1}$. These results demonstrate the existence of a variety of nontrivial compact $\lambda$-hypersurfaces.

Self-expanders have also attracted considerable attention. Smoczyk \cite{smoczyk} studied complete mean-convex self-expanders and obtained rigidity results under natural geometric assumptions. Ancari and Cheng \cite{acexp} studied volume growth, weighted volume, spectral estimates and rigidity properties of complete immersed self-expanders. In contrast with self-shrinkers, there are no closed self-expanders in Euclidean space. This naturally motivates the study of $\lambda$-self-expanders, obtained by introducing a nonzero constant into the self-expander equation.

The theory of $\lambda$-self-expanders was systematically developed by
Ancari and Cheng \cite{ac}, who first exhibited the following three basic
families of standard examples.
\begin{example}
Any affine hyperplane in $\mathbb{R}^{n+1}$ is a $\lambda$-self-expander.
More precisely, if the distance from the hyperplane to the origin is $d$,
then
\[
\lambda=\pm\frac{d}{2},
\]
where the sign depends on the choice of orientation.
\end{example}

\begin{example}
A round sphere $\mathbb{S}^n_r(0)$ centered at the origin with radius
$r>0$ is a $\lambda$-self-expander with
\[
\lambda=\frac{n}{r}+\frac{r}{2}\geq \sqrt{2n}.
\]
Consequently, for every $\lambda>\sqrt{2n}$, there are two centered round
spheres corresponding to the radii
\[
r=\lambda\pm\sqrt{\lambda^2-2n}.
\]
Moreover, among round spheres, only those centered at the origin can be
$\lambda$-self-expanders.
\end{example}

\begin{example}
For each $1\leq k\leq n-1$, the generalized cylinder
\[
\mathbb{S}^k_r(0)\times\mathbb{R}^{n-k}
\]
is a $\lambda$-self-expander with
\[
\lambda=\frac{k}{r}+\frac{r}{2}\geq \sqrt{2k}.
\]
Thus, for every $\lambda>\sqrt{2k}$, there are two such generalized
cylinders, with radii
\[
r=\lambda\pm\sqrt{\lambda^2-2k}.
\]
As in the spherical case, the spherical factor must be centered at the
origin.
\end{example}

In addition to identifying these standard examples, Ancari and Cheng \cite{ac} established several rigidity results characterizing them and studied the area growth and weighted area properties of complete $\lambda$-self-expanders. Related generalized self-similar equations were considered by L\'opez \cite{lopez}, who classified ruled and translation surfaces satisfying $H=\alpha\langle F,N\rangle+\lambda$. More recently, Deng, Yao, Yu and Zhang \cite{deng} obtained classification results for two-dimensional $\lambda$-self-expanders with constant Gaussian curvature or constant squared norm of the second fundamental form.

Compared with self-shrinkers and $\lambda$-hypersurfaces, fewer nontrivial compact examples of $\lambda$-self-expanders are known. Recently, Bai and Xia \cite{baixia}, together with Florian Johne and Lauro Silini \cite{johnsilini2026}, have independently and simultaneously proved an Alexandrov-type theorem for hypersurfaces of constant weighted mean curvature in warped product manifolds with radial density. In particular, their result implies that every smooth closed embedded $\lambda$-self-expander in Euclidean space must be a round sphere centered at the origin.

In view of the results of Bai and Xia \cite{baixia} and Florian Johne and Lauro Silini \cite{johnsilini2026}, one cannot expect to construct non-round closed embedded $\lambda$-self-expanders, in contrast with the nontrivial compact embedded $\lambda$-hypersurfaces constructed in \cite{cw2}. This naturally leads us to consider the immersed setting.

In this paper, we construct a topologically spherical immersed $\lambda$-self-expander.
\begin{theorem}\label{thm:main}
Let $n \ge 2$ and set
\begin{equation}\label{eq:Lambda-explicit}
\Lambda_n = 1000^n n^3.
\end{equation}
Then, for every $\lambda \ge \Lambda_n$, there exists a smooth closed rotational immersed $\lambda$-self-expander in $\mathbb{R}^{n+1}$ which is diffeomorphic to $\mathbb{S}^n$, is not embedded, and is not a round sphere.
\end{theorem}
Moreover, the examples obtained in Theorem~\ref{thm:main} can be chosen
to be mean convex. More precisely, we have the following consequence.

\begin{corollary}\label{cor:mean-convex}
For every sufficiently large $\lambda$, there exists a closed immersed $\lambda$-self-expander in $\mathbb R^{n+1}(n\ge 2)$ with $H>0$ which is diffeomorphic to $\mathbb S^n$ but is not a round sphere.
\end{corollary}

In particular, mean convexity does not imply spherical rigidity for
closed immersed $\lambda$-self-expanders.

The proof of Theorem~\ref{thm:main} is based on a rotational shooting argument, inspired by \cite{drugan,liwei}. A continuity argument for the initial condition, together with the symmetry of the profile equation, produces the desired closed immersed hypersurface. In contrast, for $\lambda$-hypersurfaces, the Bernstein theorem ensures the existence of the first horizontal inflection point, while no analogous rigidity theorem is available in the literature for our case. 

The paper is organized as follows. In Section~2, we give the necessary facts about  the $\lambda$-self expanders and  derive the rotational  $\lambda$-self-expander equation.  In Sections~3, we have dealt with the solution of the equation when a tends towards infinity, which is the limit circle.
In Section~4, ~5 and ~6, we study the two relevant parts of the profile curve and establish the estimates required for the shooting argument, and  we choose the shooting parameter, complete the profile by reflection and prove Theorem~\ref{thm:main}. In Section~7, we give a numerical illustrations of the shooting construction in dimension $n = 2$.

\textbf{Acknowledgement.} I wish to extend my sincere gratitude to Prof. Ruijun Wu, Prof. Tongzhu Li and Prof. Yong Luo, for introducing me to these problems and for their valuable guidance throughout our discussions.

\textbf{AI disclosure.} The authors used AI tools to help improve the English exposition and to assist with the numerical computations in Section 7. The authors have reviewed and verified all numerical results and mathematical arguments, and take full responsibility for the content of the manuscript.

\section{Preliminaries}
In this section, we give the necessary facts about  the $\lambda$-self expanders in $\mathbb{R}^{n+1}$.
Let $F: M^n\to \mathbb{R}^{n+1}$ be an $n$-dimensional immersed  hypersurface in the  Euclidean space $\mathbb{R}^{n+1}$. Let $\{e_1,\cdots,e_n\}$ be a local orthonormal frame field of $F(M^n)$ with respect to the first fundamental form $I=\langle dF,dF\rangle$ with dual coframe field $\{\omega_1,\cdots,\omega_n\}$. Let $\xi$ be the unit normal vector field of $F$. In this paper we use the following conventions on the ranges of indices:
$$ 1\leq i,j,k,l\leq n.$$
Thus we have the structure equation of $F$,
\begin{equation}\label{eq-str}
\begin{split}
&dF=\sum_{i} \omega_{i}e_{i},\\
&de_{i}=\sum_{j} \omega_{ij}e_{j}+\sum_jh_{ij}\omega_j\xi,\\
&d\xi=-\sum_{i,j}h_{ij}\omega_je_{i},
\end{split}
\end{equation}
where $h_{ij}$ denote the components of the second fundamental form of $F$
 and $\omega_{ij}$ are the Levi-Civita connection forms of $I$ with respect to $\{\omega_1,\cdots,\omega_n\}$.

Let
$$ h=\sum_{i,j} h_{ij}\omega_{i}\otimes \omega_{j} $$
be the second fundamental form, then the squared norm of the second fundamental form $S=|h|^2=tr(h^2)=\sum_{i,j}(h_{ij})^{2}, $  and the mean curvature $H=tr(h)=\sum_{i} h_{ii}.$

The Gauss equations and the Codazzi equations of the hypersurface $F$ are given by
\begin{equation}\label{eq-1}
\begin{split}
&R_{ijkl}=h_{ik}h_{jl}-h_{il}h_{jk},~~1\leq i,j,k,l\leq n,\\
&h_{ij,k}=h_{ik,j}, ~~1\leq i,j,k\leq n,
\end{split}
\end{equation}
where the covariant derivative of $h_{ij}$ is defined by
$$\sum_kh_{ij,k}\omega_k=dh_{ij}+\sum_kh_{ik}\omega_{kj}+\sum_kh_{kj}\omega_{ki}.$$

By the Gauss equation (\ref{eq-1}), we obtain the Ricci curvature $R_{ij}$ and the scalar curvature $R$ of the hypersurface,
\begin{equation}\label{ricc}
\begin{split}
&R_{ij}=Hh_{ij}-\sum_{m}h_{im}h_{mj},\\
&R=H^2-S,
\end{split}
\end{equation}

It is well known that there exist no closed self-expanders in $\mathbb{R}^{n+1}$. But the situation is different for $\lambda\neq 0$. In \cite{ac},
S. Ancari and X. Cheng proved the following results,
\begin{theorem}\cite{ac}\label{le2-2}
Let $F:M^n\to \mathbb{R}^{n+1}$ be a smooth immersed closed $\lambda$-self-expander. Then,\\
(i)~ $\lambda\geq \sqrt{2n}$. Moreover, the sphere $\mathbb{S}^n(\sqrt{2n})$ is the only closed $\lambda$-self-expander with $\lambda=\sqrt{2n}$.\\
(ii)~ If $$S\leq -\frac{1}{2}+\frac{\lambda(\lambda-\sqrt{\lambda^2-2n})}{2n},$$ then $F$ is a sphere $\mathbb{S}^n(r)$ with $r=\lambda+\sqrt{\lambda^2-2n}$.
\end{theorem}

Let $\mathbb{R}^2_+=\{(x,z)\in \mathbb{R}^2|z>0\}$ denote the upper half-plane.  Let $SO(n)$
denote the special orthogonal group and acts on $\mathbb{R}^{n+1}=\{(x,y)|x\in \mathbb{R},~~y\in \mathbb{R}^n\}$
by $$A(x,y)=(x, Ay), A\in SO(n).$$
Thus  the space of orbits $\mathbb{R}^{n+1}/SO(n)$ is
the closed upper half plane $$\overline{\mathbb{R}^2_+}=\{(x,z)\in \mathbb{R}^2|z\geq0\}.$$
Let $\pi: \mathbb{R}^{n+1}\to \overline{\mathbb{R}^2_+}$ be the the projection, $\pi(x,y)=(x,|y|)$.

If a hypersurface $M^n$ is invariant under the action of $SO(n)$, then the projection $\pi(M^n)$
give  a profile curve in the upper half plane, which can be parametrized by Euclidean arc length
and written as $\Gamma(s)=(x(s), z(s))$. Conversely, if we have a curve $\Gamma(s)=(x(s), z(s)), s\in (a,b)$
 parametrized by Euclidean arc length in $\mathbb{R}^2_+$, then we can reconstruct the hypersurface
by
\begin{equation}\label{hyper}
\begin{split}
&F: (a,b)\times \mathbb{S}^{n-1}\to \mathbb{R}^{n+1},\\
&F(s, \alpha)=(x(s), z(s)\alpha).
\end{split}
\end{equation}
Since the profile curve is parametrized by arc length, then
\begin{equation}\label{eq:arclength}
x_s^2+z_s^2=1.
\end{equation}
The inward unit normal $\xi$ of the hypersurface $F$ can be written by
\begin{equation}\label{eq:normal}
\xi=(-z_s,x_s\alpha),
\end{equation}
we have
\[
H=x_sz_{ss}-x_{ss}z_s-\frac{n-1}{z}x_s,
\qquad
\langle F, \xi \rangle=-xz_s+zx_s.
\]
Thus \eqref{def} is equivalent to
\begin{equation}\label{eq:unscaled-profile}
x_sz_{ss}-x_{ss}z_s-
\left(\frac{n-1}{z}+\frac z2\right)x_s
+\frac x2z_s=\lambda.
\end{equation}
We impose the initial data
\begin{equation}\label{eq:initial-unscaled}
x(0)=b,\qquad z(0)=0,\qquad x_s(0)=0,\qquad z_s(0)=1.
\end{equation}

Set
\begin{equation}\label{eq:scaling}
X=\lambda x,\qquad Z=\lambda z,\qquad t=\lambda s,
\qquad B=\lambda b.
\end{equation}
Then
\begin{equation}\label{eq:unit-speed}
X_t^2+Z_t^2=1.
\end{equation}

Thus the $\lambda$-self-expander equation is written by
\begin{equation}\label{eq:scaled-profile}
X_tZ_{tt}-X_{tt}Z_t-\frac{n-1}{Z}X_t
+\frac1{2\lambda^2}(XZ_t-ZX_t)=1,
\end{equation}
with the initial data
\begin{equation}\label{eq:scaled-initial}
X(0)=B,\qquad Z(0)=0,\qquad X_t(0)=0,\qquad Z_t(0)=1.
\end{equation}

Write
\begin{equation}\label{eq:angle}
X_t=\cos\theta,\qquad Z_t=\sin\theta,
\qquad \theta(0)=\frac\pi2.
\end{equation}
The signed curvature
\[
k:=X_tZ_{tt}-X_{tt}Z_t
\]
satisfies $k=\theta_t$, and \eqref{eq:scaled-profile} becomes
\begin{equation}\label{eq:k-direct}
k=1+\frac{n-1}{Z}X_t
-\frac1{2\lambda^2}(XZ_t-ZX_t).
\end{equation}
Since \eqref{eq:unit-speed} holds, then
\begin{equation}\label{eq:acceleration}
X_{tt}=-kZ_t,\qquad Z_{tt}=kX_t.
\end{equation}

\section{The limiting circle}
We consider the limiting circle as $\lambda\to \infty$, and then use perturbation methods to find the solution of \eqref{eq:unscaled-profile} near the limit circle.

As $\lambda\to\infty$, \eqref{eq:scaled-profile} formally reduces to
\begin{equation}\label{eq:limiting-equation}
X_tZ_{tt}-X_{tt}Z_t-\frac{n-1}{Z}X_t=1,
\qquad X_t^2+Z_t^2=1.
\end{equation}
A direct calculation gives the first integral
\[
\frac{d}{dt}\left(Z^{n-1}X_t+\frac{Z^n}{n}\right)=0.
\]
Since $Z(0)=0$ and $n\ge2$, the constant is zero, and therefore
\begin{equation}\label{eq:limiting-Xt}
X_t=-\frac Zn.
\end{equation}
Consequently $k=1/n$ and
\[
\theta(t)=\frac\pi2+\frac tn.
\]
Using the initial data gives the limiting solution
\begin{equation}\label{eq:limiting-solution}
X_0(t)=B-n+n\cos\frac tn,
\qquad
Z_0(t)=n\sin\frac tn.
\end{equation}
Hence
\begin{equation}\label{eq:limiting-circle}
\bigl(X_0-(B-n)\bigr)^2+Z_0^2=n^2.
\end{equation}
Thus the limiting profile is the upper semicircle of radius $n$ centered at
$(B-n,0)$, beginning at $(B,0)$ and ending at $(B-2n,0)$.

For the quantitative shooting argument we now fix once and for all
\begin{equation}\label{eq:B-range}
\frac{3n}{2}\le B\le\frac{5n}{2}.
\end{equation}
We have,
\begin{equation}\label{eq:B-positive}
B-n\ge\frac n2>0.
\end{equation}

Define, exactly as in the original shooting argument,
\begin{equation}\label{eq:P-def}
P:=Z^{n-1}X_t+\frac{Z^n}{n}.
\end{equation}
For the limiting circle $P\equiv0$.

\begin{lemma}\label{lem:P-identities}
Along every solution of \eqref{eq:scaled-profile} with $Z>0$,
\begin{align}
P_t&=\frac1{2\lambda^2}Z^{n-1}Z_t(XZ_t-ZX_t),
\label{eq:P-t}\\
k&=\frac1n+(n-1)\frac{P}{Z^n}
-\frac1{2\lambda^2}(XZ_t-ZX_t).
\label{eq:k-P}
\end{align}
If $Z_t\ne0$, then
\begin{equation}\label{eq:P-Z}
\frac{dP}{dZ}
=\frac1{2\lambda^2}Z^{n-1}(XZ_t-ZX_t).
\end{equation}
\end{lemma}

\begin{proof}
Differentiating \eqref{eq:P-def} and using
$X_{tt}=-kZ_t$ gives
\[
P_t=Z^{n-1}Z_t
\left(\frac{n-1}{Z}X_t-k+1\right),
\]
and \eqref{eq:P-t} follows from \eqref{eq:k-direct}.  Moreover,
\[
\frac{X_t}{Z}=\frac{P}{Z^n}-\frac1n,
\]
which substituted into \eqref{eq:k-direct} gives \eqref{eq:k-P}.
Equation \eqref{eq:P-Z} follows from \eqref{eq:P-t} whenever $Z_t\ne0$.
\end{proof}

We shall repeatedly use the following elementary consequences.  For every $n\ge2$ and
$\lambda\ge\Lambda_n$,
\begin{align}
\frac{10n}{\lambda^2}&\le\frac1{40n},
\label{eq:numeric-1}\\
\frac{60n^3 800^n}{\lambda^2}&\le\frac n{400},
\label{eq:numeric-2}\\
\frac{1000n^4(1600n)^n}{\lambda^2}
&\le\frac{n^{n-1/2}}{128},
\label{eq:numeric-3}\\
\frac{7n^{n+2}}{\lambda^2}&\le 8^{-n},
\label{eq:numeric-4}\\
\frac{15n^{n+1}}{\lambda^2}&\le 16^{-(n-1)}.
\label{eq:numeric-5}
\end{align}
These inequalities are very far from sharp; they are included only to make
all later thresholds explicit.

\section{The first part of the profile curve}

Since $Z_t(0)=1$, the profile curve can be written near the initial point as
a graph
\[
X=\gamma(Z).
\]
Then
\[
Z_t=\frac1{\sqrt{1+\gamma'^2}},\qquad
X_t=\frac{\gamma'}{\sqrt{1+\gamma'^2}},\qquad
k=-\frac{\gamma''}{(1+\gamma'^2)^{3/2}},
\]
and \eqref{eq:scaled-profile} becomes
\begin{equation}\label{eq:gamma}
\frac{\gamma''}{1+\gamma'^2}
=-\frac{n-1}{Z}\gamma'-\sqrt{1+\gamma'^2}
+\frac1{2\lambda^2}(\gamma-Z\gamma').
\end{equation}
The initial conditions are
\begin{equation}\label{eq:gamma-initial}
\gamma(0)=B,\qquad \gamma'(0)=0.
\end{equation}

In \cite{drugan}, Drugan have obtained the existence of solutions of self-shrinker near $z = 0$ by discussing the existence, uniqueness and continuous dependence of solutions on initial height at $z$ near $0$. That's also true for $\lambda$-self-expander.

\begin{lemma}
    When $\frac{3n}{2}\le B\le\frac{5n}{2}$, we have $X_{tt}(0)<0$ for sufficiently large $\lambda$.
\end{lemma}
\begin{proof}
    Since $X(t)=\gamma(Z(t))$, $Z_t(0)=1$ and $\gamma'(0)=0$, Substituting these into (\ref{eq:gamma}) yields 
\begin{equation}\label{eq:Xtt0}
X_{tt}(0)=\gamma''(0)
=\frac{B/(2\lambda^2)-1}{n}<0,
\end{equation}
because $B\le5n/2$ and $\lambda\ge\Lambda_n$.
\end{proof}
Thus the profile initially bends to the left. Let $t_1$ be the first time, if it exists, such that
\begin{equation}\label{eq:t1-def}
Z_t(t_1)=0,\qquad X_t(t_1)=-1,
\end{equation}
or equivalently $\theta(t_1)=\pi$. In the case of $\lambda$-hypersurfaces, the existence of $t_1$ is guaranteed by the Bernstein theorem \cite{cw3}. While no such analogous result has yet been established for $\lambda$-self-expanders, we establish the existence of $t_1$ here.

\begin{proposition}\label{prop:first-branch}
For every $n\ge2$, every $\lambda\ge\Lambda_n$, and every $B$ satisfying $\frac{3n}{2}\le B\le\frac{5n}{2}$, the time $t_1$ exists.  On $(0,t_1)$,
\begin{equation}\label{eq:first-signs}
X_t<0,\qquad Z_t>0,\qquad k>0,
\end{equation}
and hence $X_{tt}<0$.  Moreover,
\begin{align}
\left|t_1-\frac{n\pi}{2}\right|
&\le\frac{10\pi n^3}{\lambda^2},
\label{eq:t1-explicit}\\
|Z(t_1)-n|
&\le\frac{10\pi n^3}{\lambda^2},
\label{eq:Z1-explicit}\\
|X(t_1)-(B-n)|
&\le\frac{10\pi n^3}{\lambda^2}.
\label{eq:X1-explicit}
\end{align}
In particular,
\begin{equation}\label{eq:t1-coarse}
\frac{9n}{10}<Z(t_1)<\frac{11n}{10},
\qquad X(t_1)>\frac{2n}{5}.
\end{equation}
\end{proposition}

\begin{proof}
Consider the maximal interval on which
\[
\frac\pi2\le\theta<\pi,\qquad Z>0,\qquad t\le n\pi.
\]
Since the curve has unit speed and $\frac{3n}{2}\le B\le \frac{5n}{2}$, then
\[
|X|\le B+t<6n,\qquad 0<Z\le t<4n.
\]
Hence
\begin{equation}\label{eq:Q-first}
|XZ_t-ZX_t|\le |X|+Z<10n.
\end{equation}
Because $Z_t>0$, \eqref{eq:P-Z} and $P(0)=0$ imply
\begin{equation}\label{eq:P-first-explicit}
|P(t)|
\le\frac{10n}{2\lambda^2}
\int_0^{Z(t)}r^{n-1}\,dr
=\frac{5}{\lambda^2}Z(t)^n.
\end{equation}
Substituting \eqref{eq:Q-first} and \eqref{eq:P-first-explicit} into
\eqref{eq:k-P},
\begin{equation}\label{eq:k-first-explicit}
\left|k-\frac1n\right|
\le\frac{5(n-1)}{\lambda^2}+\frac{5n}{\lambda^2}
<\frac{10n}{\lambda^2}.
\end{equation}
By \eqref{eq:numeric-1},
\[
k\ge\frac1n-\frac1{40n}>\frac1{2n}.
\]
Thus $\theta_t=k>0$.  If $\theta$ did not reach $\pi$ before $t=n\pi$,
then
\[
\theta(n\pi)\ge\frac\pi2+\frac{n\pi}{2n}=\pi,
\]
a contradiction.  Therefore $t_1$ exists and \eqref{eq:first-signs} holds.
Moreover $X_{tt}=-kZ_t<0$.

Since $k\ge \frac{1}{2n}$, \eqref{eq:k-first-explicit} gives
\begin{equation}\label{eq:inverse-k-first}
\left|\frac1k-n\right|
=\frac{n|k-\frac{1}{n}|}{k}
\le\frac{20n^3}{\lambda^2}.
\end{equation}
Since $\frac{dt}{d\theta}=\frac1k$, integrating from
$\frac{\pi}{2}$ to $\pi$ yields $t_1=\int_{\frac{\pi}{2}}^\pi\frac1k\,d\theta.$ On the other hand, $\frac{n\pi}{2}=\int_{\frac{\pi}{2}}^\pi n\,d\theta.$
Subtracting these two identities and taking absolute values,
we obtain
\begin{equation*}
\left|t_1-\frac{n\pi}{2}\right| \le\int_{\frac{\pi}{2}}^\pi \left|\frac1k-n\right|d\theta
\le\frac{20n^3}{\lambda^2} \int_{\frac{\pi}{2}}^\pi d\theta
=\frac{10\pi n^3}{\lambda^2}.
\end{equation*}
Similarly, we obtain
\begin{align*}
|Z(t_1)-n|
&\le\int_{\frac{\pi}{2}}^{\pi}
|\sin\theta|\left|\frac1k-n\right|\,d\theta
\le\frac{10\pi n^3}{\lambda^2},\\
|X(t_1)-(B-n)|
&\le\int_{\frac{\pi}{2}}^{\pi}
|\cos\theta|\left|\frac1k-n\right|\,d\theta
\le\frac{10\pi n^3}{\lambda^2}.
\end{align*}
Since $\lambda\ge\Lambda_n\ge20n$, the right-hand side is less than
$\frac{n}{10}$.  This yields \eqref{eq:t1-coarse}.
\end{proof}
Proposition \ref{prop:first-branch} provides quantitative estimates for the location of the first horizontal tangent, showing that $t_1$, $Z(t_1)$, and $X(t_1)$ are close to $\frac{n\pi}{2}$, $n$, and $B-n$, respectively.
On the first branch $Z_t>0$, so it may be written as $X=\gamma(Z)$ and
\[
\gamma'(Z)=\frac{X_t}{Z_t}<0,
\qquad
\gamma'(Z)\longrightarrow-\infty
\quad\text{as }t\to t_1.
\]

\section{The second part of the profile curve}

Immediately after $t_1$, one has $Z_t<0$.  The second branch can therefore be written locally as $X=\beta(Z)$.  We first move to a fixed angle before the singular axis, then use the first nonzero variation of $P$, then reach the vertical tangent, and finally continue to the second horizontal tangent.

For the quantitative argument we fix
\begin{equation}\label{eq:delta0}
\delta_0=\frac1{100}.
\end{equation}
The choice $\delta_0=\frac{1}{100}$ is sufficient to ensure that all the estimates required in the subsequent arguments hold, although it is conservative and not optimal.
Let $t_{\delta_0}$ denote the first time, if it exists, such that
\begin{equation}\label{eq:tdelta-def}
\theta(t_{\delta_0})=\frac{3\pi}{2}-\delta_0.
\end{equation}
For the limiting circle this angle is reached at $t=n(\pi-\delta_0),$ and the corresponding point is $X_0=B-n-n\cos\delta_0$, $Z_0=n\sin\delta_0.$
We now establish quantitative estimates near $t_{\delta_0}$, analogous to those in Proposition \ref{prop:first-branch}.
\begin{lemma}\label{lem:fixed-angle}
For every $\lambda\ge\Lambda_n$ and $\frac{3n}{2}\le B\le\frac{5n}{2}.$ the time $t_{\delta_0}$ is well-defined.  Moreover,
\begin{align}
\left|t_{\delta_0}-n(\pi-\delta_0)\right|
&\le\frac{60n^3 800^n}{\lambda^2},
\label{eq:tdelta-explicit}\\
\left|X(t_{\delta_0})-\bigl(B-n-n\cos\delta_0\bigr)\right|
&\le\frac{60n^3 800^n}{\lambda^2},
\label{eq:Xdelta-explicit}\\
\left|Z(t_{\delta_0})-n\sin\delta_0\right|
&\le\frac{60n^3 800^n}{\lambda^2}.
\label{eq:Zdelta-explicit}
\end{align}
In particular,
\begin{equation}\label{eq:fixed-coarse}
\left|X(t_{\delta_0})-\bigl(B-n-n\cos\delta_0\bigr)\right|
\le\frac n{400},
\end{equation}
\begin{equation}\label{eq:Zdelta-upper}
0<Z(t_{\delta_0})\le2n\sin\delta_0\le\frac n{50},
\end{equation}
and
\begin{equation}\label{eq:k-fixed}
k\ge\frac{3}{4n}
\end{equation}
on $[t_1,t_{\delta_0}]$.
\end{lemma}
\begin{proof}
Starting from $t_1$, consider the maximal connected interval on which
\[
\pi\le\theta\le\frac{3\pi}{2}-\delta_0,
\qquad
Z\ge\frac n{400},
\qquad
k\ge\frac1{2n}.
\]
It is nonempty by Proposition~\ref{prop:first-branch}.  On this interval $Z_t<0$, so $Z\le Z(t_1)<\frac{11n}{10}$.  Also $\frac{dt}{d\theta}=\frac{1}{k}\leq 2n,$ hence for every $t$ in the interval,
\begin{equation}
    t-t_1=\int_{\pi}^{\theta(t)}\frac{1}{k}\,d\theta \leq 2n\bigl(\theta(t)-\pi\bigr) \leq 2n\left(\frac{\pi}{2}-\delta_0\right)<n\pi.
\end{equation}
By Proposition \ref{prop:first-branch} , $\frac{3n}{2}\le B\le\frac{5n}{2}$ and $\lambda\ge\Lambda_n$, we have
\begin{equation*}
|X(t_1)|\leq |X(t_1)-(B-n)|+|B-n| \leq \frac{10\pi n^3}{\lambda^2}+\frac{3n}{2}<\frac{n}{10}+\frac{3n}{2}=\frac{8n}{5}.
\end{equation*}
then we obtain
\begin{align*}
|X(t)|&=\left|X(t_1)+\int_{t_1}^{t}X_s(s)\,ds\right| \leq |X(t_1)|+\int_{t_1}^{t}|X_s(s)|\,ds\\
&\leq|X(t_1)|+(t-t_1)<\frac{8n}{5}+n\pi <5n.
\end{align*}
Finally, using $|X_t|\leq 1$ and $|Z_t|\leq 1$,
we conclude that
\begin{align}
|XZ_t-ZX_t|&\leq |X||Z_t|+Z|X_t|\notag \leq |X|+Z\notag\\
&<5n+\frac{11n}{10}
<7n.
\label{eq:Q-fixed}
\end{align}
At $t_1$, \eqref{eq:P-first-explicit} gives
\[
|P(t_1)|\le\frac5{\lambda^2}
\left(\frac{11n}{10}\right)^n.
\]
Since $Z$ decreases after $t_1$, \eqref{eq:P-Z} and
\eqref{eq:Q-fixed} give
\[
|P(t)-P(t_1)|
\le\frac{7n}{2\lambda^2}
\int_{Z(t)}^{Z(t_1)}r^{n-1}\,dr
\le\frac{7}{2\lambda^2}
\left(\frac{11n}{10}\right)^n.
\]
Consequently,
\begin{equation}\label{eq:P-fixed}
|P(t)|
\le\frac{10(2n)^n}{\lambda^2}.
\end{equation}
Because $Z\ge \frac{n}{400}$, equations \eqref{eq:k-P},
\eqref{eq:Q-fixed} and \eqref{eq:P-fixed} yield
\begin{equation}\label{eq:k-fixed-error}
\left|k-\frac1n\right|
\le\frac{11n800^n}{\lambda^2}\le \frac{1}{4n},
\end{equation}
so in fact $k\ge\frac{3}{4n}$. 

Furthermore, combining with \eqref{eq:k-fixed-error}, we obtain
\begin{equation}\label{eq:inverse-k-fixed}
\left|\frac1k-n\right|
=\frac{n}{k}\left|k-\frac1n\right| \leq
n\cdot\frac{4n}{3}\cdot\frac{11n\cdot 800^n}
{\lambda^2}  =\frac{44n^3\cdot 800^n}
{3\lambda^2}  \leq\frac{15n^3\cdot 800^n}
{\lambda^2}.
\end{equation}
Since $\frac{dt}{d\theta}=\frac1k$, integrating from $\theta=\pi$ to an arbitrary
$\theta\in[\pi,\frac{3\pi}{2}-\delta_0]$, we obtain
\[
t-t_1=\int_\pi^\theta\frac1k\,d\varphi.
\]
Consequently, Proposition~\ref{prop:first-branch} and the preceding
estimate yield
\[
\begin{aligned}
\left|t-n\left(\theta-\frac{\pi}{2}\right)\right|
&\leq \left|t_1-\frac{n\pi}{2}\right|
+\int_\pi^\theta
\left|\frac1k-n\right|\,d\varphi\\
&\leq \frac{10\pi n^3}{\lambda^2}
+\frac{15n^3 800^n}{\lambda^2}(\theta-\pi)\\
&\leq \frac{60n^3 800^n}{\lambda^2}.
\end{aligned}
\]
In particular, taking
$\theta=\frac{3\pi}{2}-\delta_0$, we obtain
\[
\left|t_{\delta_0}-n(\pi-\delta_0)\right|
\leq \frac{60n^3 800^n}{\lambda^2}.
\]
As functions of the angle $\theta$, the limiting circle is
\begin{equation}\label{eq:limit-angle}
X_0=B-n+n\sin\theta,
\qquad
Z_0=-n\cos\theta.
\end{equation}
In particular,
\[
X_0(\pi)=B-n,\qquad Z_0(\pi)=n.
\]
By Proposition~\ref{prop:first-branch}, the errors at
$\theta=\pi$ satisfy
\begin{equation}\label{eq:initial-errors-fixed}
\left|X(t_1)-X_0(\pi)\right|
\leq\frac{10\pi n^3}{\lambda^2},
\qquad
\left|Z(t_1)-Z_0(\pi)\right|
\leq\frac{10\pi n^3}{\lambda^2}.
\end{equation}

Since $\theta_t=k>0$, we may use $\theta$ as a parameter.
Writing $X(\theta)=X(t(\theta))$ and
$Z(\theta)=Z(t(\theta))$, we have
\[
\frac{dX}{d\theta}=\frac{\cos\theta}{k},
\qquad
\frac{dZ}{d\theta}=\frac{\sin\theta}{k}.
\]
On the other hand,
\[
\frac{dX_0}{d\theta}=n\cos\theta,
\qquad
\frac{dZ_0}{d\theta}=n\sin\theta.
\]
Consequently, subtracting the corresponding expressions above and integrating from $\theta=\pi$ to $\theta$, we obtain
\begin{align*}
X(\theta)-X_0(\theta)
&=X(t_1)-(B-n)
 +\int_\pi^\theta
 \cos\varphi\left(\frac1k-n\right)d\varphi,\\
Z(\theta)-Z_0(\theta)
&=Z(t_1)-n
 +\int_\pi^\theta
 \sin\varphi\left(\frac1k-n\right)d\varphi.
\end{align*}
Here, $k$ in the integrals is evaluated at $t(\varphi)$.

Using \eqref{eq:initial-errors-fixed} and
\eqref{eq:inverse-k-fixed}, we obtain
\begin{align*}
|X(\theta)-X_0(\theta)|
&\leq\frac{10\pi n^3}{\lambda^2}
 +\frac{15n^3 800^n}{\lambda^2}
\int_\pi^\theta|\cos\varphi|\,d\varphi\leq\frac{(10\pi+15\cdot800^n)n^3}{\lambda^2}
 \leq\frac{60n^3 800^n}{\lambda^2},\\
|Z(\theta)-Z_0(\theta)|
&\leq\frac{10\pi n^3}{\lambda^2}
 +\frac{15n^3 800^n}{\lambda^2}
\int_\pi^\theta|\sin\varphi|\,d\varphi\leq\frac{(10\pi+15\cdot800^n)n^3}{\lambda^2}
 \leq\frac{60n^3 800^n}{\lambda^2}.
\end{align*}
Adding these two inequalities yields
\begin{equation}\label{eq:fixed-trajectory}
|X-X_0|+|Z-Z_0|
\leq\frac{120n^3 800^n}{\lambda^2}.
\end{equation}


In particular, by \eqref{eq:numeric-2}, each of the two coordinate errors is
at most $\frac{n}{400}$.  Since
\[
Z_0\ge n\sin\delta_0>\frac n{200},
\]
so we obtain $Z>\frac{n}{400}$.  Thus the trajectory reaches the target angle and $t_{\delta_0}$ exists. The same integration gives \eqref{eq:fixed-coarse}-- \eqref{eq:k-fixed}.  Finally,
$\frac{n}{400}\le \tfrac12 n\sin\delta_0$, so
\eqref{eq:Zdelta-upper} follows.
\end{proof}

The quantity $P$ vanishes identically on the limiting circle.  We now make its first nonzero variation quantitative.

For $0<\delta<\frac{\pi}{4}$ define
\begin{equation}\label{eq:A-delta}
A_\delta(B)=\frac{n^n}{2}\int_0^{\pi-\delta}\sin^{n-1}\varphi\,\cos\varphi\bigl[n+(B-n)\cos\varphi\bigr]d\varphi.
\end{equation}
Then
\begin{equation}\label{eq:A0}
A_0(B)=\frac{n^n}{2}(B-n)\int_0^\pi\sin^{n-1}\varphi\cos^2\varphi\,d\varphi.
\end{equation}
Because $B-n\ge \frac{n}{2}$, the minimum of $A_0(B)$ is
\begin{equation}\label{eq:m0}
m_0:=\frac{n^{n+1}}4
\int_0^\pi\sin^{n-1}\varphi\cos^2\varphi\,d\varphi.
\end{equation}
We shall need the elementary lower bound
\begin{equation}\label{eq:beta-lower}
\int_0^\pi\sin^{n-1}\varphi\cos^2\varphi\,d\varphi
\ge\frac1{8n^{\frac{3}{2}}}.
\end{equation}
Indeed, with $u=\cos\varphi$ the left-hand side equals $2\int_0^1 u^2(1-u^2)^{\frac{n-2}{2}}\,du.$
Restricting the integration to
$u\in\left[\frac{1}{2\sqrt{n}},\frac{1}{\sqrt{n}}\right]$
and using
$\left(1-\frac{1}{n}\right)^{\frac{n-2}{2}}\geq\frac{1}{2}$
proves \eqref{eq:beta-lower}. Hence
\begin{equation}\label{eq:m0-lower}
m_0\geq\frac{n^{n-\frac{1}{2}}}{32}.
\end{equation}

\begin{lemma}\label{lem:first-variation}
For every $\lambda\ge\Lambda_n$ and $\frac{3n}{2}\le B\le\frac{5n}{2}$,
\begin{equation}\label{eq:Pdelta-lower}
\frac{m_0}{2\lambda^2}
\le P(t_{\delta_0})
\le\frac{6n^{n+1}}{\lambda^2}.
\end{equation}
\end{lemma}

\begin{proof}
First we compare $A_{\delta_0}$ with $A_0$.  Since
$|B-n|\le \frac{3n}{2}$,
\begin{align*}
|A_0(B)-A_{\delta_0}(B)|
&\le\frac{n^n}{2}\frac{5n}{2}
\int_{\pi-\delta_0}^{\pi}\sin^{n-1}\varphi\,d\varphi\\
&\le\frac54 n^n\delta_0^n.
\end{align*}
For $\delta_0=\frac{1}{100}$ and $n\ge2$,
\[
\frac54n^n100^{-n}\le\frac{n^{n-\frac{1}{2}}}{128}
\le\frac{m_0}{4}.
\]
Therefore
\begin{equation}\label{eq:Adelta-positive}
A_{\delta_0}(B)\ge\frac{3m_0}{4}.
\end{equation}

Next, by \eqref{eq:P-t}, using $\theta$ as parameter,
\begin{equation}\label{eq:P-angle-integral}
\lambda^2P(t_{\delta_0})
=\frac12\int_{\pi/2}^{3\pi/2-\delta_0}
Z^{n-1}\sin\theta
(X\sin\theta-Z\cos\theta)\frac{d\theta}{k}.
\end{equation}
For the limiting circle, \eqref{eq:limit-angle} and $\frac{1}{k}=n$ give exactly
$A_{\delta_0}(B)$.

On the whole angle interval under consideration,
Proposition~\ref{prop:first-branch} and Lemma~\ref{lem:fixed-angle} give
\[
|X-X_0|+|Z-Z_0|
\le\frac{120n^3 800^n}{\lambda^2},
\qquad
\left|\frac1k-n\right|
\le\frac{20n^3 800^n}{\lambda^2}.
\]
Moreover $|X|,|X_0|\le3n$, $0\le Z,Z_0\le2n$, and therefore
\[
|X\sin\theta-Z\cos\theta|\le6n,
\qquad
|X_0\sin\theta-Z_0\cos\theta|\le6n.
\]
Using
\[
|Z^{n-1}-Z_0^{n-1}|
\le(n-1)(2n)^{n-2}|Z-Z_0|
\]
and expanding the difference of the two integrands in
\eqref{eq:P-angle-integral}, one obtains
\begin{equation}\label{eq:first-var-error}
\left|\lambda^2P(t_{\delta_0})-A_{\delta_0}(B)\right|
\le\frac{1000n^4(1600n)^n}{\lambda^2}.
\end{equation}
By \eqref{eq:numeric-3} and \eqref{eq:m0-lower}, the right-hand side is at
most $\frac{m_0}{4}$.  Combining this with \eqref{eq:Adelta-positive} gives
\[
\lambda^2P(t_{\delta_0})\ge\frac{m_0}{2}.
\]
For the upper bound, the absolute value of \eqref{eq:A-delta} is at most
$5n^{n+1}$, while the error in \eqref{eq:first-var-error} is much smaller
than $n^{n+1}$.  Hence
\[
\lambda^2P(t_{\delta_0})\le6n^{n+1}.
\]
This proves \eqref{eq:Pdelta-lower}.
\end{proof}

\begin{lemma}\label{lem:P-positive}
Along the connected part of the trajectory satisfying
\begin{equation}\label{eq:small-angle-sector}
\frac{3\pi}{2}-\delta_0\le\theta\le\frac{3\pi}{2},
\qquad Z>0,
\end{equation}
the total variation of $P$ satisifies
\begin{equation}\label{eq:P-variation-explicit}
|\Delta P|\le\frac{m_0}{8\lambda^2}.
\end{equation}
Consequently,
\begin{equation}\label{eq:P-positive-explicit}
P\ge\frac{m_0}{4\lambda^2}>0
\end{equation}
on that connected segment.
\end{lemma}

\begin{proof}
By \eqref{eq:fixed-coarse},
$|X(t_{\delta_0})|<0.52n$.  By \eqref{eq:small-angle-sector},
\[
\left|\frac{dX}{dZ}\right|=|\cot\theta|
\le\tan\delta_0<1.
\]
Also, by \eqref{eq:Zdelta-upper},
$0<Z\le \frac{n}{50}$.  Hence $|X|<0.55n$ throughout this segment and
\begin{equation}\label{eq:Q-small}
|XZ_t-ZX_t|<n.
\end{equation}
Using \eqref{eq:P-Z},
$$
|\Delta P|\le\frac{n}{2\lambda^2} \int_0^{Z(t_{\delta_0})}r^{n-1}\,dr \le\frac1{2\lambda^2}\left(\frac n{50}\right)^n.
$$
For $n\ge2$,
\[
\frac12\left(\frac n{50}\right)^n
\le\frac{n^{n-1/2}}{256}
\le\frac{m_0}{8},
\]
where the last inequality follows from \eqref{eq:m0-lower}.  This proves
\eqref{eq:P-variation-explicit}.  Subtracting the maximum possible loss from
\eqref{eq:Pdelta-lower} gives
\[
P\ge\frac{m_0}{2\lambda^2}-\frac{m_0}{8\lambda^2}
\ge\frac{m_0}{4\lambda^2}.
\]
\end{proof}

\begin{proposition}\label{prop:vertical-tangent}
For every $\lambda\ge\Lambda_n$ and $\frac{3n}{2}\le B\le\frac{5n}{2}$, there exists a unique finite time $t_m>t_1$ such that
\begin{equation}\label{eq:tm-def}
X_t(t_m)=0,\qquad Z_t(t_m)=-1,\qquad Z(t_m)>0.
\end{equation}
Moreover,
\begin{equation}\label{eq:k-to-tm}
k>0\qquad\text{on }[t_1,t_m],
\end{equation}
and
\begin{equation}\label{eq:Zm-explicit}
\frac{nm_0}{4\lambda^2}
\le Z(t_m)^n
\le\frac{7n^{n+2}}{\lambda^2}.
\end{equation}
In particular,
\begin{equation}\label{eq:Zm-small}
Z(t_m)\le\frac18.
\end{equation}
\end{proposition}

\begin{proof}
By Lemma~\ref{lem:fixed-angle}, the profile curve reaches
$\theta=\frac{3\pi}{2}-\delta_0$.  Starting there, consider the maximal connected
segment satisfying \eqref{eq:small-angle-sector} with
$\theta<\frac{3\pi}{2}$.  By Lemma~\ref{lem:P-positive},
\[
P\ge\frac{m_0}{4\lambda^2}>0.
\]
Since $X_t=\cos\theta<0$ before $\frac{3\pi}{2}$,
\[
P=Z^{n-1}\left(X_t+\frac Zn\right)<\frac{Z^n}{n}.
\]
Hence
\begin{equation}\label{eq:Z-away-axis}
Z^n\ge\frac{nm_0}{4\lambda^2}>0.
\end{equation}
Thus the trajectory cannot hit $Z=0$ before reaching the vertical tangent.

The bounds in the proof of Lemma~\ref{lem:P-positive} give
$|XZ_t-ZX_t|<n$.  Therefore, from \eqref{eq:k-P} and $P>0$,
\[
k\ge\frac1n-\frac{n}{2\lambda^2}\ge\frac1{2n}.
\]
Hence $\theta$ reaches $\frac{3\pi}{2}$ in finite time; call that time $t_m$.
Strict positivity of $k$ gives uniqueness.  At $t_m$,
\begin{equation}\label{eq:Ptm}
P(t_m)=\frac{Z(t_m)^n}{n}.
\end{equation}
The lower bound in \eqref{eq:Zm-explicit} follows from
\eqref{eq:P-positive-explicit}.  For the upper bound, Lemma
\ref{lem:first-variation} and Lemma~\ref{lem:P-positive} give
\[
P(t_m)\le\frac{6n^{n+1}}{\lambda^2}
+\frac{m_0}{8\lambda^2}
\le\frac{7n^{n+1}}{\lambda^2}.
\]
Together with \eqref{eq:Ptm}, this proves the upper bound in
\eqref{eq:Zm-explicit}.  Finally, \eqref{eq:numeric-4} implies
$Z(t_m)^n\le8^{-n}$, hence \eqref{eq:Zm-small}.
\end{proof}

On $(t_1,t_m)$ one has $Z_t<0$ and $k>0$, so
\begin{equation}\label{eq:convex-second-left}
X_{tt}=-kZ_t>0.
\end{equation}
Thus $X_t$ increases strictly from $-1$ at $t_1$ to $0$ at $t_m$.

\begin{proposition}\label{prop:second-horizontal}
For every $\lambda\ge\Lambda_n$ and every $B$ satisfying
\eqref{eq:B-range}, there exists a finite time $t_2>t_m$ such that
\begin{equation}\label{eq:t2-def}
Z_t(t_2)=0,\qquad X_t(t_2)=1,\qquad Z(t_2)>0.
\end{equation}
Moreover,
\begin{equation}\label{eq:k-after-tm}
k>0\qquad\text{on }[t_m,t_2],
\end{equation}
and
\begin{equation}\label{eq:horizontal-displacement}
0\le X(t_2)-X(t_m)\le\frac14.
\end{equation}
\end{proposition}

\begin{proof}
At $t=t_m$,
$X_t=0$ and $Z_t=-1$, so \eqref{eq:k-direct} gives
\[
k(t_m)=1+\frac{X(t_m)}{2\lambda^2}.
\]
From the fixed-angle estimate and the slope bound in the small-height sector,
$|X(t_m)|<n$.  Hence $k(t_m)>0$.

Let $[t_m,T)$ be the maximal interval on which
\begin{equation}\label{eq:bootstrap}
\frac{3\pi}{2}\le\theta<2\pi,
\qquad Z>0,
\qquad |X-X(t_m)|<1.
\end{equation}
On this interval $Z_t<0$, so $0<Z\le Z(t_m)\le \frac{1}{8}$.  Also
$|X|<n+1\le \frac{3n}{2}$, hence
\begin{equation}\label{eq:Q-after-tm}
|XZ_t-ZX_t|<2n.
\end{equation}
Using $Z$ as parameter in \eqref{eq:P-Z},
\[
|P(t)-P(t_m)|
\le\frac{2n}{2\lambda^2}
\int_0^{Z(t_m)}r^{n-1}\,dr
=\frac{Z(t_m)^n}{\lambda^2}.
\]
Since $P(t_m)=\frac{Z(t_m)^n}{n}$,
\[
\frac{|P(t)-P(t_m)|}{P(t_m)}\le\frac n{\lambda^2}<\frac12.
\]
Therefore
\begin{equation}\label{eq:P-after-tm}
\frac12P(t_m)\le P(t)\le\frac32P(t_m).
\end{equation}
In particular $P(t)>0$.  Equation \eqref{eq:k-P} and
\eqref{eq:Q-after-tm} give
\[
k\ge\frac1n-\frac n{\lambda^2}\ge\frac1{2n}.
\]
Thus $\theta$ is strictly increasing.

The curve cannot reach $Z=0$, because
\[
\lim_{Z \to 0}P=Z^{n-1}\left(X_t+\frac Zn\right)=0,
\]
contradicting \eqref{eq:P-after-tm}.
It remains to rule out failure of the bootstrap through
$|X-X(t_m)|=1$.

If $\frac{3\pi}{2}\le\theta\le\frac{7\pi}{4}$, then $0\le X_t\le-Z_t$, and hence
\begin{equation}\label{eq:first-right-displacement}
0\le X(t)-X(t_m)
\le Z(t_m)-Z(t)
\le Z(t_m)\le\frac18.
\end{equation}

If $\frac{7\pi}{4}\le\theta<2\pi$, then $X_t\ge \frac{1}{\sqrt2}$.  By \eqref{eq:P-after-tm},
\[
\frac1{\sqrt2}Z^{n-1}
\le P(t)
\le\frac32P(t_m)
=\frac{3}{2n}Z(t_m)^n.
\]
Using the upper bound in \eqref{eq:Zm-explicit},
\begin{equation}\label{eq:Z-last-sector}
Z^{n-1}\le\frac{15n^{n+1}}{\lambda^2}.
\end{equation}
By \eqref{eq:numeric-5},
\begin{equation}\label{eq:Z-last-small}
Z\le\frac1{16}.
\end{equation}
Moreover, from \eqref{eq:k-direct}, \eqref{eq:Q-after-tm} and
$X_t\ge1/\sqrt2$,
\[
k
\ge1+\frac{n-1}{\sqrt2 Z}-\frac n{\lambda^2}
\ge\frac{n-1}{\sqrt2 Z}.
\]
Therefore
\[
0\le\frac{dX}{d\theta}=\frac{X_t}{k}
\le\frac{\sqrt2}{n-1}Z\le\sqrt2 Z.
\]
Integrating over an angular interval of length $\frac{\pi}{4}$ and using
\eqref{eq:Z-last-small},
\begin{equation}\label{eq:last-right-displacement}
0\le X(t)-X(t_{\frac{7\pi}{4}})
\le\frac{\pi\sqrt2}{4}\cdot\frac1{16}<\frac18.
\end{equation}
Combining \eqref{eq:first-right-displacement} and
\eqref{eq:last-right-displacement},
\[
0\le X(t)-X(t_m)<\frac14.
\]
Hence the bootstrap condition $|X-X(t_m)|<1$ cannot fail.  The solution
therefore continues until $\theta=2\pi$.  This defines $t_2$, proves
\eqref{eq:t2-def} and \eqref{eq:k-after-tm}, and also gives
\eqref{eq:horizontal-displacement}.
\end{proof}

On $(t_m,t_2)$, $Z_t<0$ and $k>0$, so again
\[
X_{tt}=-kZ_t>0.
\]
Together with \eqref{eq:convex-second-left}, the whole second branch is
convex to the right.

\begin{proposition}\label{prop:endpoint}
For every $\lambda\ge\Lambda_n$ and $\frac{3n}{2}\le B\le\frac{5n}{2}$ ,
\begin{equation}\label{eq:endpoint-explicit}
|X(t_2)-(B-2n)|<\frac n4.
\end{equation}
\end{proposition}

\begin{proof}
At $t_{\delta_0}$, Lemma~\ref{lem:fixed-angle} gives
\[
\left|X(t_{\delta_0})-\bigl(B-n-n\cos\delta_0\bigr)\right|
\le\frac n{400}.
\]
From $t_{\delta_0}$ to $t_m$,
$|X_t|\le-Z_t$, hence
\[
|X(t_m)-X(t_{\delta_0})|
\le Z(t_{\delta_0})-Z(t_m)
\le Z(t_{\delta_0})
\le\frac n{50}.
\]
From $t_m$ to $t_2$, Proposition~\ref{prop:second-horizontal} gives
\[
0\le X(t_2)-X(t_m)\le\frac14\le\frac n8.
\]
Therefore
\begin{align*}
|X(t_2)-(B-2n)|
&\le n(1-\cos\delta_0)
+\frac n{400}+\frac n{50}+\frac n8\\
&\le\frac n{20000}+\frac n{400}+\frac n{50}+\frac n8
<\frac n4,
\end{align*}
where we used $1-\cos\delta_0\le \frac{\delta_0^2}{2}$ and
$\delta_0=\frac{1}{100}$.
\end{proof}

\section{The shooting argument and proof of the main theorem}

\begin{proof}[Proof of Theorem~\ref{thm:main}]
For each
\[
B\in\left[\frac{3n}{2},\frac{5n}{2}\right]
\]
and each $\lambda\ge\Lambda_n$, Proposition~\ref{prop:second-horizontal}
produces a second horizontal tangent $t_2=t_2(B)$ with
\[
Z(t_2)>0,\qquad (X_t,Z_t)(t_2)=(1,0).
\]
By Proposition~\ref{prop:endpoint}, when $B=\frac{3n}{2}$,
\[
X(t_2)
<\left(\frac{3n}{2}-2n\right)+\frac n4
=-\frac n4<0,
\]
while when $B=\frac{5n}{2}$,
\[
X(t_2)
>\left(\frac{5n}{2}-2n\right)-\frac n4
=\frac n4>0.
\]
Therefore, there must exist a \(B\) such that \(X(t_2)=0\). Furthermore, at \(t_2\), we have 
\begin{equation}\label{eq:shooting-zero}
X(t_2)=0,
\qquad Z(t_2)>0,
\qquad (X_t,Z_t)(t_2)=(1,0).
\end{equation}
Reflect the profile segment $[0,t_2]$ across the $Z$-axis.  Because the curve meets the $Z$-axis orthogonally at $(0,Z(t_2))$, the reflected segment joins smoothly.  Rotating the profile about the $X$-axis gives a smooth closed immersed hypersurface diffeomorphic to $\mathbb{S}^n$.

It remains to prove that the immersion is not embedded.  By
\eqref{eq:t1-coarse},
\[
X(t_1)>\frac{2n}{5}>0.
\]
On $(t_m,t_2)$, $X_t>0$, while $X(t_2)=0$, so $X(t_m)<0$.  Since
$X_t<0$ on $(t_1,t_m)$, there is a unique
$t_0\in(t_1,t_m)$ such that
\[
X(t_0)=0.
\]
Since $Z_t<0$ on $(t_1,t_2)$,
\[
Z(t_0)>Z(t_2)>0.
\]
Thus the profile crosses the $Z$-axis at $(0,Z(t_0))$ before reaching the
reflection point $(0,Z(t_2))$.  Its reflected branch passes through the same point.  If $\theta_0=\theta(t_0)\in(\pi,\frac{3\pi}{2})$, the two tangent vectors are $(\cos\theta_0,\sin\theta_0)$ and $(\cos\theta_0,-\sin\theta_0)$, whose determinant is nonzero. Hence the self-intersection is transverse. The resulting $\lambda$-self-expander is therefore immersed but not embedded, and in particular it is not a round sphere.
\end{proof}

\begin{proof}[Proof of Corollary~\ref{cor:mean-convex}]
Along the closed scaled profile constructed above, the estimates of the
previous sections give
\[
|X(t)|+|Z(t)|\le5n.
\]
Indeed, on the first branch $|X|\le \frac{5n}{2}$ and $Z<\frac{11n}{10}$; on the second
branch the fixed-angle, small-height and bootstrap estimates are even
smaller, and the reflected branch has the same bounds.  Scaling back,
\[
|F_\lambda|
=\frac1\lambda\sqrt{X^2+Z^2}
\le\frac{5n}{\lambda}.
\]
Therefore
\[
\bigl|\langle F_\lambda, N\rangle\bigr|
\le |F_\lambda|
\le \frac{5n}{\lambda}.
\]
From \eqref{def},
\[
H=\lambda+\frac{1}{2}\langle F_\lambda, N\rangle
\ge\lambda-\frac{5n}{2\lambda}.
\]
Since $\lambda\ge\Lambda_n$, hence $H>0$.
\end{proof}

\begin{remark}
The threshold
\[
\Lambda_n = 1000^n n^3
\]
is only a sufficient one.  It is intentionally chosen to dominate, with a large safety margin, the powers of $n$ and the fixed-angle denominators that arise when $P/Z^n$ is estimated near the small-height region.  Considerable improvement is possible by optimizing the fixed angle, replacing several coarse bounds by sharp ones, and tracking the first variation more closely. No such optimization is needed for the existence theorem.
\end{remark}

\section{A two-dimensional example}

We illustrate the shooting construction in dimension $n=2$
with a parameter satisfying the explicit hypothesis of
Theorem~\ref{thm:main}. Choose
\[
n=2,\qquad \lambda=10^7.
\]
Since
\[
\Lambda_2=1000^2\cdot 2^3
=8\times10^6<10^7,
\]
the quantitative shooting construction applies.

For $n=2$, the limiting solution is
\[
X_0(t)=B-2+2\cos\frac t2,
\qquad
Z_0(t)=2\sin\frac t2.
\]
Its natural closing value is $B=4$. For finite $\lambda$,
however, the closing parameter is determined by shooting.

The scaled profile equations become
\begin{equation}\label{eq:2d-1e7-system}
\begin{cases}
X_t=\cos\theta,\\
Z_t=\sin\theta,\\[1mm]
\displaystyle
\theta_t
=1+\frac{\cos\theta}{Z}
-\frac{X\sin\theta-Z\cos\theta}{2\times10^{14}},
\end{cases}
\end{equation}
with initial data
\[
X(0)=B,\qquad Z(0)=0,\qquad
\theta(0)=\frac{\pi}{2}.
\]

For each $B\in[3,5]$, let $t_2(B)$ denote the second
horizontal tangent, characterized by
\[
\theta(t_2(B))=2\pi.
\]
Define the shooting function
\[
F(B)=X_B(t_2(B)).
\]
By Proposition~\ref{prop:endpoint},
\[
|F(B)-(B-4)|<\frac14.
\]
In particular,
\[
F(3)<-\frac34<0,
\qquad
F(5)>\frac34>0.
\]
Since $F$ is continuous, there exists $B_*\in(3,5)$
such that
\begin{equation}\label{eq:2d-1e7-closing}
F(B_*)=0,
\qquad
Z_{B_*}(t_2(B_*))>0.
\end{equation}
Applying the same endpoint estimate at $B=B_*$ yields
\[
|B_*-4|<\frac14,
\qquad
\frac{15}{4}<B_*<\frac{17}{4}.
\]
Consequently, the initial point in the original
variables satisfies
\[
b_*=\frac{B_*}{10^7},
\qquad
3.75\times10^{-7}<b_*<4.25\times10^{-7}.
\]

We now give a numerical illustration of the resulting
profile. In dimension two, the auxiliary quantity is
\[
P=Z\cos\theta+\frac{Z^2}{2}.
\]
To avoid loss of accuracy near the axis, we integrate
using $\theta$ as the independent variable:
\[
\frac{dX}{d\theta}=\frac{\cos\theta}{k},
\qquad
\frac{dP}{d\theta}
=
\frac{Z\sin\theta
(X\sin\theta-Z\cos\theta)}
{2\lambda^2k},
\qquad
\frac{dt}{d\theta}=\frac1k,
\]
where
\[
k=\frac12+\frac{P}{Z^2}
-\frac{X\sin\theta-Z\cos\theta}{2\lambda^2}.
\]
The trajectory is initialized near $\theta=\pi/2$
using the regular axis expansion. Solving
$F(B)=0$ numerically gives
\begin{equation}\label{eq:2d-1e7-numerical-B}
B_*\approx3.9999999999992,
\qquad
b_*\approx3.9999999999992\times10^{-7}.
\end{equation}

The distinguished points of the scaled profile,
listed in chronological order, are approximately
\begin{align*}
&t=0:
&& (X,Z)=(3.9999999999992,0),\\
&t_1\approx3.14159265359:
&& (X,Z)\approx
   (1.9999999999992,\;2.0000000000000),\\
&t_{\mathrm{cross}}\approx6.2831839:
&& (X,Z)\approx
   (0,\;1.4\times10^{-6}),\\
&t_m\approx6.28318507624:
&& (X,Z)\approx
   (-4.3\times10^{-13},\;2.30940108\times10^{-7}),\\
&t_2\approx6.28318530718:
&& (X,Z)\approx
   (0,\;2.66666667\times10^{-14}).
\end{align*}
Here $t_{\mathrm{cross}}$ is the first crossing of
the $Z$-axis, $t_1$ is the first horizontal tangent,
and $t_m$ is the subsequent vertical tangent.
These coordinates are numerical approximations;
the existence and closing properties follow
independently from the quantitative shooting theorem.

As a consistency check, at the vertical tangent
$X_t(t_m)=0$, and hence
\[
P(t_m)=\frac{Z(t_m)^2}{2}.
\]
The numerical calculation gives
\[
\lambda^2P(t_m)\approx\frac83.
\]
This agrees with the leading first-variation value
\[
\frac{2^2}{2}(B_*-2)
\int_0^\pi
\sin\varphi\cos^2\varphi\,d\varphi
=\frac43(B_*-2)
\approx\frac83.
\]

Moreover, the quantitative estimate
\eqref{eq:Zm-explicit} gives the rigorous bounds
\[
\frac{2}{3\lambda^2}
\le Z(t_m)^2
\le\frac{112}{\lambda^2}.
\]
In particular, the vertical tangent remains
strictly above the axis.

Returning to the original variables,
\[
x=\frac X\lambda,\qquad z=\frac Z\lambda,
\]
the second horizontal tangent is located
numerically at approximately
\[
(x(t_2),z(t_2))
\approx(0,\;2.66666667\times10^{-21}).
\]

By \eqref{eq:2d-1e7-closing}, the actual profile
meets the $Z$-axis orthogonally at a strictly
positive height. Reflecting it across the
$Z$-axis and rotating the resulting closed
profile about the $x$-axis produces a smooth
closed immersed $\lambda$-self-expander in
$\mathbb R^3$ which is diffeomorphic to $\mathbb{S}^2$
and is not embedded.

\section*{Statements and Declarations}
\noindent \textbf{Conflict of interest}
The author declares that there is no conflict of interest.

\noindent \textbf{Funding}
The author declares that no funds, grants, or other financial support were received for this work.

\end{document}